\documentclass[10pt, twoside]{article}
\usepackage{graphicx}
\usepackage{amsmath}
\usepackage{amsfonts}
\usepackage{amssymb}
\usepackage{amsthm} % when in use, disable the definition of the proof env.
\usepackage{enumerate}
\usepackage[cp1250]{inputenc}
\usepackage[a4paper, left=3cm, right=3cm, top=4cm, bottom=4cm, headsep=0.5cm]{geometry}

\usepackage[misc,geometry]{ifsym} % for envelope symbol \Letter

\usepackage{caption}
\usepackage{color}
\definecolor{sepia}{cmyk}{0, 0.83, 1, 0.70}

\usepackage[plainpages=false,pdfpagelabels,bookmarksnumbered,%
 colorlinks=true,%
 linkcolor=sepia,%
 citecolor=sepia,%
 unicode]{hyperref}

\usepackage{todonotes}

\usepackage{fancyhdr}
\newcommand\shorttitle{The Optimal Error Bound for the Method of Simultaneous Projections}
\newcommand\shortauthors{S. Reich and R. Zalas}

\newtheorem{theorem}{Theorem}

\newtheorem{corollary}[theorem]{Corollary}

\newtheorem{lemma}[theorem]{Lemma}
\newtheorem{fact}[theorem]{Fact}

\theoremstyle{definition}

\newtheorem{definition}[theorem]{Definition}
\newtheorem{example}[theorem]{Example}

\newtheorem{remark}[theorem]{Remark}

\makeatletter
\renewenvironment{proof}[1][\proofname]
{\par
	\pushQED{$\blacksquare$} % originally "\qed" was instead of $\blacksquare$
	\normalfont\topsep6\p@\@plus6\p@\relax
	\trivlist
	\item[\hskip\labelsep\bfseries#1\@addpunct{.}]
	\ignorespaces}
{\popQED \endtrivlist\@endpefalse}
\makeatother

\DeclareMathOperator*{\fix}{Fix}

\begin{document}

\title{\vspace{-5em}\textbf{ \Large The Optimal Error Bound\\ for the Method of Simultaneous Projections}}

\author{Simeon Reich\thanks{Department of Mathematics, The Technion -- Israel Institute of Technology, 3200003 Haifa, Israel}
\thanks{\Letter \ sreich@tx.technion.ac.il} \and
Rafa\l\ Zalas\footnotemark[1]
\thanks{\Letter \ zalasrafal@gmail.com}}

\maketitle

\begin{abstract}
In this paper we find the optimal error bound (smallest possible estimate, independent of the starting point) for the linear convergence rate of the simultaneous projection method applied to closed linear subspaces in a real Hilbert space. We achieve this by computing the norm of an error operator which we also express in terms of the Friedrichs number. We compare our estimate with the optimal one provided for the alternating projection method by Kayalar and Weinert (1988). Moreover, we relate our result to the alternating projection formalization of Pierra (1984) in a product space. Finally, we adjust our results to closed affine subspaces and put them in context with recent dichotomy theorems.


%\footnote{\\
%Submit to \textit{Journal of Approximation Theory}, Editorial Board: H.H. Bauschke, P.L. Combettes and F. Deutsch.\\ \\
%\textbf{MSC}: 41 Approximations and expansions\\
%41A25  	Rate of convergence, degree of approximation\\
%41A28  	Simultaneous approximation\\
%41A44  	Best constants\\
%41A65  	Abstract approximation theory (approximation in normed linear spaces and other
%abstract spaces)}

\end{abstract}

\section{Introduction}\label{sec:intro}
Let $M_1,\ldots,M_r$ be closed and linear subspaces of a real Hilbert space $\mathcal H$ and let $M:=\bigcap_{i=1}^r M_i$. By $P_{C}$ we denote the metric projection onto a nonempty, closed and convex set $C\subseteq\mathcal H$. In this paper we consider simultaneous and cyclic projection methods. The following two theorems are well known:
\begin{theorem}[von Neumann \cite{Neumann1949} and Halperin \cite{Halperin1962}]\label{th:NeumannHalperin}
For each $x\in\mathcal H$,
\begin{equation}
\lim_{k\rightarrow\infty}\left\|\left(P_{M_r}\ldots P_{M_1}\right)^k (x)-P_M(x)\right\| = 0.
\end{equation}
\end{theorem}

\begin{theorem}[Lapidus \cite{Lapidus1981} and Reich \cite{Reich1983}]\label{th:LapidusReich}
For each $x\in\mathcal H$,
\begin{equation}
\lim_{k\rightarrow\infty}\left\|\left(\frac 1 r \sum_{i=1}^r P_{M_i}\right)^k (x)-P_M(x)\right\|= 0.
\end{equation}
\end{theorem}
Let $T,T^\infty\colon\mathcal H\rightarrow\mathcal H$ be such that $T^k(x)$ converges to $T^\infty (x)$ for every $x\in\mathcal H$. Following \cite{BauschkeDeutschHundal2009, DeutschHundal2010}, we say that $T^k$ \textit{converges arbitrarily slowly} to $T^\infty$ if for every sequence $\{a_k\}_{k=0}^\infty \subseteq (0,\infty)$ satisfying $a_k\rightarrow 0$ as $k\rightarrow\infty$, there is $x\in\mathcal H$ such that $\|T^k(x)-T^\infty (x)\|\geq a_k$ for every $k=0,1,2,\ldots$. We also recall that $T^k$ \textit{converges linearly} to $T^\infty$ if for some $c,f(x)>0$ and $q\in (0,1)$, we have $\|T^k(x)-T^\infty(x)\|\leq cq^k f(x)$ for all $k=0,1,2,\ldots$. In this paper $T$ and $T^\infty$ are related to projections onto linear or affine subspaces in which case we use $f(x)=\|x\|$ or $f(x)=\|x-T^\infty (0)\|$, respectively. Note that even in the case of closed linear subspaces the convergence in Theorems \ref{th:NeumannHalperin} and \ref{th:LapidusReich} does not have to be linear and moreover, it may indeed be arbitrarily slow. To see this, we now quote a relevant dichotomy result.

\begin{theorem}[Bauschke, Deutsch and Hundal \cite{BauschkeDeutschHundal2009, DeutschHundal2010}]\label{th:dichotomy}
Let $T:=P_{M_r}\ldots P_{M_1}$ or $T:=\frac 1 r \sum_{i=1}^r P_{M_i}$. Then exactly one of the following two statements holds:
\begin{enumerate}[(i)]
  \item $\sum_{i=1}^r M_i^\perp$ is closed. Then $T^k$ converges linearly to $P_M$.
  \item $\sum_{i=1}^r M_i^\perp$ is not closed. Then $T^k$ converges arbitrarily slowly to $P_M$.
\end{enumerate}
\end{theorem}
Alternative (ii) of the above theorem has recently been extended in the case of the cyclic projection method. Following \cite{BadeaSeifert2016}, we say that $T^k$ \textit{converges super-polynomially fast} to $T^\infty$ on a nonempty set $X\subseteq\mathcal H$ if $\|T^k(x)-T^\infty(x)\|/k^{-\alpha} \rightarrow 0$ as $k\rightarrow\infty$ for all $\alpha > 0$ and $x\in X$.

\begin{theorem}[Badea and Seifert \cite{BadeaSeifert2016}]\label{th:dichotomyBS}
Let $T:=P_{M_r}\ldots P_{M_1}$. If $\sum_{i=1}^r M_i^\perp$ is not closed, then $T^k$ converges super-polynomially fast to $P_M$ on some dense linear subspace $X\subseteq\mathcal H$.
\end{theorem}

We remark here that Theorem \ref{th:dichotomyBS} follows from \cite[Theorem 4.3]{BadeaSeifert2016} which was only proved for a complex Hilbert space. In order to see this, one can apply a complexification argument which has kindly been provided to us by Catalin Badea and David Seifert; see the Appendix for more details.\ For more dichotomy and trichotomy results concerning arbitrarily slow convergence, we refer the interested reader to \cite{BadeaGrivauxMuller2011, DeutschHundal2011, DeutschHundal2015}. Note that using the above theorems, one can easily see that arbitrarily slow as well as super-polynomially fast convergence may only happen in the infinite dimensional case.

A very natural question related to Theorem \ref{th:dichotomy} (i) is the following one: What is the optimal error bound (smallest possible estimate, independent of $x$) such that $\|T^k(x)-P_M(x)\|\leq cq^k\|x\|$? This question can be answered by finding the norm of the \textit{error operator} $T^k-P_M$. In the case of the alternating projection method ($r=2$), we have, by Aronszajn \cite{Aronszajn1950} (inequality), and Kayalar and Weinert \cite{KayalarWeinert1988} (equality),
\begin{equation}\label{eq:AKW}
\|(P_{M_2}P_{M_1})^k-P_M\|=\cos(M_1,M_2)^{2k-1},
\end{equation}
where by
\begin{equation}\label{eq:def:cos2}
\cos(M_1,M_2):=\sup\{|\langle x_1,x_2\rangle| \mid x_i\in M_i\cap (M_1\cap M_2)^\perp, \|x_i\|\leq 1\}\in[0,1]
\end{equation}
we denote the cosine of the Friedrichs angle between the subspaces $M_1$ and $M_2$. In addition, one can show that $\cos(M_1,M_2)<1$ if and only if $M_1^\perp+M_2^\perp$ is closed; see, for example,  \cite[Theorem 9.35]{Deutsch2001} and \cite[p. 235]{Deutsch2001} for a complete proof and detailed historical notes going back to \cite{BauschkeBorwein1993, Deutsch1986} and Simoni\v c. As far as we are aware, for $r>2$ the exact computation of the error operator norm for both algorithmic operators is still unknown; see, for example, \cite{BadeaGrivauxMuller2011, BadeaSeifert2016, DeutschHundal2011, DeutschHundal2015, PustylnikReichZaslavski2012, PustylnikReichZaslavski2013}. Even for $r=2$, the norm $\|((P_{M_1}+P_{M_2})/2)^k-P_M\|$ seems to be unknown. Note that one could try to find the optimal estimate for the simultaneous projection method by using \eqref{eq:AKW} and the corresponding alternating projection formalization of Pierra \cite{Pierra1984} in the product space $\mathcal H^r$. This approach, although very natural, is somewhat misleading and, when applied directly, provides a weaker result than the optimal one; compare Example \ref{ex:Pierra1}, Theorem \ref{th:norm} and Example \ref{ex:Pierra2} below.

The main contribution of this paper is to extend Theorem \ref{th:dichotomy} (i) in the case of the simultaneous projection method by finding the optimal error bound, that is, by computing the exact value of $\|(\frac 1 r \sum_{i=1}^r P_{M_i})^k-P_M\|$ for any $r\geq 2$; see Theorems \ref{th:norm} and \ref{th:dichLin} below. For $r=2$ we show that this norm is greater than the norm in \eqref{eq:AKW}, which somewhat explains why, in general, the alternating projection method is indeed faster than its simultaneous variant whenever we have linear convergence; see Remark \ref{rem:alternatingPM}.

In addition, we formally extend Theorem \ref{th:dichotomyBS} for the simultaneous projection method with $T=\frac 1 r \sum_{i=1}^r P_{M_i}$ by using the alternating projection formalization in a product space.

Finally, by using a translation argument, we obtain analogous results in the case of affine subspaces; see Corollary \ref{th:dichAff}.

\section{Main result}
We begin this section with a simple example showing that a direct application of Pierra's alternating projection formalization in a product space indeed leads to linear convergence, but the obtained estimate, as we show in Theorem \ref{th:norm} below, is not optimal.

\begin{example}[Alternating projection formalization of Pierra ]\label{ex:Pierra1}
Let $M_1,\ldots,M_r\subseteq\mathcal H$ be closed and linear subspaces and $M:=\bigcap_{i=1}^r M_i$. Moreover, following Pierra \cite{Pierra1984}, we consider the subsets
\begin{equation}\label{eq:CD}
\mathbf C:=M_1\times\ldots\times M_r \quad \text{and} \quad \mathbf D:=\{\mathbf x=(x,\ldots,x)\mid x\in \mathcal H\}
\end{equation}
of the product space $\mathcal H^r$ equipped with the scalar product $\langle \mathbf x,\mathbf y\rangle:=\frac 1 r \sum_{i=1}^r\langle x_i, y_i\rangle$, where $\mathbf x=(x_1,\ldots,x_r)$, $\mathbf y=(y_1,\ldots,y_r)$, and $x_i,y_i\in\mathcal H$, $i=1,\ldots,r$. We recall \cite[Fact 3.2, Lemma 3.3]{DeutschHundal2008} that for any $\mathbf x=(x,\ldots,x)\in\mathbf D$, we have
\begin{equation}\label{eq:Pierra}
  \|\mathbf x\|=\|x\|,\quad
  (P_{\mathbf D}P_{\mathbf C})^k (\mathbf x)=\left(T^k (x),\ldots, T^k (x)\right) \quad \text{and}\quad
  P_{\mathbf C\cap\mathbf D} (\mathbf x)= (P_M(x),\ldots,P_M(x)),
\end{equation}
where $T:=\frac 1 r \sum_{i=1}^rP_{M_i}$. Moreover, by \eqref{eq:AKW}, $\|(P_{\mathbf D}P_{\mathbf C})^k-P_{\mathbf C \cap \mathbf D}\|=\cos(\mathbf C, \mathbf D)^{2k-1}.$ This leads to the following estimate:
\begin{align}
\nonumber
  \|T^k(x) - P_M(x)\|
  &= \|(T^k(x)-P_M(x),\ldots, T^k(x)-P_M(x))\|\\
  \nonumber
  &= \|(P_{\mathbf D}P_{\mathbf C})^k (\mathbf x) - P_{\mathbf C\cap\mathbf D} (\mathbf x)\|\\
  \nonumber
  &\leq \cos(\mathbf C, \mathbf D)^{2k-1} \|\mathbf x\|\\
  &= \cos(\mathbf C, \mathbf D)^{2k-1} \|x\|.
\end{align}
Consequently, if $(P_{\mathbf D}P_{\mathbf C})^k$ converges linearly to $P_{\mathbf C\cap\mathbf D}$, that is, if $\cos(\mathbf C, \mathbf D)<1$, then $T^k$ converges linearly to $P_M$. On the other hand, $\mathbf y=P_{\mathbf D}P_{\mathbf C}\mathbf x\in\mathbf D$ for every $\mathbf x \in \mathcal H^r$ and we have
\begin{equation}
\|(P_{\mathbf D}P_{\mathbf C})^k (\mathbf x) - P_{\mathbf C\cap\mathbf D} (\mathbf x)\|
=\|(P_{\mathbf D}P_{\mathbf C})^{k-1} (\mathbf y) - P_{\mathbf C\cap\mathbf D}(\mathbf y)\|
=\|T^{k-1}(y) - P_M(y)\|.
\end{equation}
Thus $(P_{\mathbf D}P_{\mathbf C})^k$ converges linearly to $P_{\mathbf C\cap\mathbf D}$ whenever $T^k$ converges linearly to $P_M$.

\end{example}

We now prove the following general lemma. A closely related result can be found in \cite[Theorem 2.18]{BauschkeCruzETAL2016}.

\begin{lemma}\label{th:basic}
Let $\mathcal H$ be a real Hilbert space, $T\in B(\mathcal H)$ and let $M\subseteq F=\fix T$ be a nonempty, closed and linear subspace. Assume that $P_M=P_MT$ which holds, for example, if $P_F=P_FT$, or $T$ is self-adjoint, or $\|T\|\leq 1$. Then
\begin{equation}\label{th:basic:1}
T^k - P_M = (T-P_M)^k
\end{equation}
and therefore
\begin{equation}\label{th:basic:2}
\|T^k - P_M\|\leq \|T-P_M\|^k.
\end{equation}
If, in addition, $T$ is normal, that is, $T^*T=TT^*$, then $T-P_M$ is normal too and consequently,
\begin{equation}\label{th:basic:3}
\|T^k - P_M\| = \|T-P_M\|^k.
\end{equation}
\end{lemma}

\begin{proof}
Note that by assumption, $P_M=TP_M=P_MT$ and we can apply the binomial theorem to obtain
\begin{equation}\label{proof:th:norm:1}
(T-P_M)^k = \sum_{l=0}^k \binom k l (-1)^l T^{k-l} P_M^{l}
= T^k + \sum_{l=1}^k \binom k l (-1)^l P_M= T^k-P_M.
\end{equation}
Thus \eqref{th:basic:2} follows. The operator $P_M$ is self-adjoint and hence $T-P_M$ is normal. We recall that for any normal $N\in B(\mathcal H)$, $\|N^k\|=\|N\|^k$; see, for example, \cite[Lemma 8.32]{Deutsch2001}. Thus equality \eqref{th:basic:3} follows from \eqref{th:basic:1}.

We now show that $P_M=P_MT$ follows from $P_F=P_FT$. Observe that $F$ is a closed linear subspace. Indeed, due to the continuity of $F$, for every $F\ni x^k \rightarrow x$, we get $x=\lim x^k=\lim T(x^k)=T(x)$. Consequently, since $M\subseteq F$ are both closed linear subspaces, we have $P_M=P_MP_F$; see \cite[Lemma 9.2]{Deutsch2001}. This implies that $P_MT=P_MP_FT=P_MP_F=P_M$.

In the next step we show that $P_F=P_FT$ holds for any self-adjoint $T$. To this end, we recall that by the characterization of the orthogonal projection \cite[Theorem 4.9]{Deutsch2001}, $y=P_F(x)$ if an only if $y\in F$ and $\langle x-y,z\rangle=0$ for every $z\in F$. Now note that $P_FT(x)\in F$ and moreover,
\begin{equation}
\langle x-P_FT(x),z\rangle=\langle x-P_FT(x),T(z)\rangle=\langle T(I-P_F)T(x),z\rangle=\langle (I-P_F)T(x),z\rangle=0,
\end{equation}
which completes this part of the proof.

Finally, we show that when $\|T\|\leq 1$, then the identity $P_F=P_FT$ also holds. In this case $\|(I+T+\ldots+T^{k-1})/k\|\leq 1$ and $\|T^k(x)/k\|\leq \frac 1 k \rightarrow 0$. Consequently, by the mean ergodic theorem \cite[Corollary VIII.5.4]{DunfordSchwartz1988}, we have
\begin{equation}
P_F(x)=\lim_k \frac {x+T(x)+\ldots+T^{k-1}(x)} k= \lim_k \frac {T(x)+T(T(x))\ldots+T^{k-1}(T(x))} k =P_FT(x),
\end{equation}
which completes the proof.
\end{proof}

Before formulating our next result, following \cite[Definition 3.2]{BadeaGrivauxMuller2011}, we recall the following generalization of the cosine of the Friedrichs angle for more than two subspaces.

\begin{definition}
Let $M_1,\ldots,M_r\subseteq\mathcal H$ be closed linear subspaces and let $M:=\bigcap_i M_i$. The \textit{Friedrichs number} is defined by
\begin{equation}\label{eq:def:Friedrich}
  \cos(M_1,\ldots,M_r):=\sup\left\{\frac{1}{r-1}\frac{\sum_{i\neq j}\langle x_i,x_j\rangle }{\sum_{i=1}^r\|x_i\|^2}\mid x_i\in M_i\cap M^{\perp}\text{ and } \sum_{i=1}^r\|x_i\|^2\neq 0\right\}.
  \end{equation}
\end{definition}
The above definition coincides in the case of $r=2$ with \eqref{eq:def:cos2}; see \cite[Lemma 3.1]{BadeaGrivauxMuller2011}. Moreover, $\cos(M_1,\ldots,M_r)\in[0,1]$. Thus we can indeed extend the notion of the Friedrichs angle $\theta\in[0,\pi/2]$ with the implicit definition $\cos(\theta)=\cos(M_1,\ldots,M_r)$.

\begin{theorem}[Exact norm value]\label{th:norm}
Let $M_1,\ldots,M_r\subseteq\mathcal H$ be closed linear subspaces and let $M:=\bigcap_i M_i$. Moreover, let $\mathbf C, \mathbf D \subseteq \mathcal H^r$ be defined as in \eqref{eq:CD}. Then, for every $k=1,2,\ldots$, we have

\begin{align}\label{eq:th:norm}
\nonumber
\left\|\left(\frac 1 r \sum_{i=1}^r P_{M_i}\right)^k - P_M\right\|
&= \left\|\frac 1 r \sum_{i=1}^r P_{M_i} - P_M\right\|^k\\
\nonumber
&= \left(\frac{r-1}{r}\cos(M_1,\ldots,M_r)+\frac 1 r\right)^k\\
\nonumber
&= \cos(\mathbf C, \mathbf D)^{2k}\\
\nonumber
&= \|P_{\mathbf D}P_{\mathbf C}P_{\mathbf D}-P_{\mathbf C\cap\mathbf D}\|^k\\
&= \|(P_{\mathbf D}P_{\mathbf C}P_{\mathbf D})^k-P_{\mathbf C\cap\mathbf D}\|.
\end{align}
\end{theorem}

\begin{proof}
Note that $T:=\frac 1 r \sum_{i=1}^rP_{M_i}$ is self-adjoint, $\|T\|\leq 1$ and $\fix T=M$. Thus the first equality follows from Lemma \ref{th:basic}. The last equality again follows from Lemma \ref{th:basic}, but this time\ applied to $\mathbf T:=P_{\mathbf D}P_{\mathbf C}P_{\mathbf D}$. Furthermore, we see that
\begin{align}\label{pr:norm:1}
\nonumber
  \|P_{\mathbf D}P_{\mathbf C}P_{\mathbf D}-P_{\mathbf C\cap\mathbf D}\|
  &=\|P_{\mathbf D}P_{\mathbf C}P_{\mathbf C}P_{\mathbf D}-P_{\mathbf C\cap\mathbf D}\|\\
  \nonumber
  &=\|(P_{\mathbf D}P_{\mathbf C}-P_{\mathbf C\cap\mathbf D})
  (P_{\mathbf C}P_{\mathbf D}-P_{\mathbf C\cap\mathbf D})\|\\
  \nonumber
  &=\|(P_{\mathbf D}P_{\mathbf C}-P_{\mathbf C\cap\mathbf D})
  (P_{\mathbf D}P_{\mathbf C}-P_{\mathbf C\cap\mathbf D})^*\|\\
  \nonumber
  &=\|P_{\mathbf D}P_{\mathbf C}-P_{\mathbf C\cap\mathbf D}\|^2\\
  &= \cos(\mathbf C, \mathbf D)^2,
\end{align}
where the second equality follows from $P_{\mathbf C\cap\mathbf D}=P_{\mathbf C\cap\mathbf D}P_{\mathbf C}=P_{\mathbf C\cap\mathbf D}P_{\mathbf D}$ and the latter one follows from \eqref{eq:AKW}.
On the other hand, $P_{\mathbf D}(\mathbf B)= \mathbf D \cap \mathbf B$, where $\mathbf B$ is the unit ball in $\mathcal H^r$. This, when combined with \eqref{eq:Pierra}, leads to
\begin{align}\label{pr:norm:2} \nonumber
  \|P_{\mathbf D}P_{\mathbf C}P_{\mathbf D}-P_{\mathbf C\cap\mathbf D}\|
  & =\|P_{\mathbf D}P_{\mathbf C}P_{\mathbf D}-P_{\mathbf C\cap\mathbf D}P_{\mathbf D}\| \\ \nonumber
  & =\sup\left\{ \|P_{\mathbf D}P_{\mathbf C}P_{\mathbf D}(\mathbf x)-P_{\mathbf C\cap\mathbf D}P_{\mathbf D}(\mathbf x)\| \mid \mathbf x \in \mathbf B \right\} \\ \nonumber
  & = \sup\left\{ \|P_{\mathbf D}P_{\mathbf C}(\mathbf y)-P_{\mathbf C\cap\mathbf D}(\mathbf y)\| \mid \mathbf y \in P_{\mathbf D}(\mathbf B) \right\}\\ \nonumber
  & = \sup\left\{ \|P_{\mathbf D}P_{\mathbf C}(\mathbf y)-P_{\mathbf C\cap\mathbf D}(\mathbf y)\| \mid \mathbf y \in \mathbf D \cap \mathbf B \right\}\\ \nonumber
  & = \sup\left\{ \|T(y)-P_M(y)\| \mid y \in \mathcal H \text{ and } \|y\|\leq 1 \right\}\\
  & = \|T-P_M\|.
\end{align}
In order to complete the proof we show that $\cos(\mathbf C, \mathbf D)^2=\frac{r-1}r\cos(M_1,\ldots,M_r)+ \frac 1 r$, where we follow the argument from the proof of \cite[Proposition 3.6]{BadeaGrivauxMuller2011}. Indeed, by the definition of $\mathbf C$ and $\mathbf D$, we have $\mathbf x=(x_1,\ldots,x_r) \in \mathbf C\cap (\mathbf C\cap \mathbf D)^\perp$ if and only if each $x_i\in M_i\cap M^\perp$. Moreover, $\mathbf y=(y,\ldots,y) \in \mathbf D\cap (\mathbf C\cap \mathbf D)^\perp$ if and only if $y\in M^\perp$. Consequently,
\begin{align}\label{pr:norm:3}\nonumber
  \cos(\mathbf C, \mathbf D)^2
  & = \sup\left\{
  \frac{|\langle \mathbf x, \mathbf y\rangle|^2}{\|\mathbf x\|^2 \|\mathbf y\|^2} \mid \mathbf x \in \mathbf C\cap (\mathbf C\cap \mathbf D)^\perp,
  \mathbf y \in \mathbf D\cap (\mathbf C\cap \mathbf D)^\perp \text{ and } \mathbf x, \mathbf y\neq 0
  \right\}\\ \nonumber
  & = \sup\left\{
  \frac{|\langle \frac 1 r \sum_{i=1}^r x_i, y\rangle|^2}{\frac 1 r \sum_{i=1}^r\|x_i\|^2 \| y\|^2} \mid x_i \in M_i\cap M^\perp, \sum_{i=1}^r \|x_i\|^2\neq 0 \text{ and }
   y \in M^\perp, y\neq 0  \right\} \\ \nonumber
  & = \sup\left\{
  \frac{|\langle \frac 1 r \sum_{i=1}^r x_i, y\rangle|^2}{\frac 1 r \sum_{i=1}^r\|x_i\|^2 \| y\|^2} \mid x_i \in M_i\cap M^\perp, \sum_{i=1}^r \|x_i\|^2\neq 0 \text{ and }
   y \in \mathcal H, y\neq 0  \right\}\\ \nonumber
  & = \sup\left\{
  \frac{\| \sum_{i=1}^r x_i\|^2}{r \sum_{i=1}^r\|x_i\|^2 } \mid x_i \in M_i\cap M^\perp,\  \sum_{i=1}^r \|x_i\|^2\neq 0 \right\} \\
  & = \frac{r-1}r\cos(M_1,\ldots,M_r)+ \frac 1 r,
\end{align}
where the last equality follows from
\begin{align}\label{}
  \frac{\| \sum_{i=1}^r x_i\|^2}{r \sum_{i=1}^r\|x_i\|^2 }
  & = \frac{\| \sum_{i=1}^r x_i\|^2-\sum_{i=1}^r\|x_i\|^2}{r \sum_{i=1}^r\|x_i\|^2 }+\frac 1 r
   = \frac{r-1}{r} \cdot \frac{\sum_{i\neq j}\langle x_i, x_j\rangle}{(r-1)\sum_{i=1}^r\|x_i\|^2} + \frac 1 r.
\end{align}
This completes the proof.
\end{proof}

\begin{remark}
Observe that, by \eqref{pr:norm:1}--\eqref{pr:norm:3}, we have
\begin{equation}\label{BGM}
\left\|\frac 1 r \sum_{i=1}^r P_{M_i} - P_M\right\|
=\cos(\mathbf C, \mathbf D)^{2}
=\frac{r-1}{r}\cos(M_1,\ldots,M_r)+\frac 1 r.
\end{equation}
The above equalities also follow from \cite[Proposition 3.6 and Proposition 3.7]{BadeaGrivauxMuller2011}. However, our proof for the first equality in \eqref{BGM} differs from the one presented in \cite{BadeaGrivauxMuller2011}.
\end{remark}

\begin{example}[Example \ref{ex:Pierra1} revisited]\label{ex:Pierra2}
In the setting of Example \ref{ex:Pierra1}, a direct application of Pierra's formalization in a product space leads to an estimate which, in view of Theorem \ref{th:norm}, is not the optimal one. The remedy to this problem is to consider $P_{\mathbf D}P_{\mathbf C}P_{\mathbf D}$ instead of $P_{\mathbf D}P_{\mathbf C}$. Indeed, for any $\mathbf x=(x,\ldots,x)\in \mathbf D$, we have $(P_{\mathbf D}P_{\mathbf C}P_{\mathbf D})^k (\mathbf x)=(P_{\mathbf D}P_{\mathbf C})^k(\mathbf x)$ and consequently,
\begin{equation}
  \|T^k(x) - P_M(x)\|
  = \|(P_{\mathbf D}P_{\mathbf C}P_{\mathbf D})^k (\mathbf x) - P_{\mathbf C\cap\mathbf D} (\mathbf x)\|
  \leq \cos(\mathbf C, \mathbf D)^{2k} \|\mathbf x\|
  = \cos(\mathbf C, \mathbf D)^{2k} \|x\|.
\end{equation}
Although the above inequality recovers the optimal error bound from Theorem \ref{th:norm}, it does not explain why this estimate is optimal.
\end{example}

\begin{remark}[Two subspaces]\label{rem:alternatingPM}
Let $M_1,M_2\subseteq\mathcal H$ be closed linear subspaces and let $M:=M_1\cap M_2$. By \eqref{eq:AKW} and Theorem \ref{th:norm},
\begin{align}\nonumber
  \| (P_{M_2}P_{M_1})^k - P_M\|&
  = \cos(M_1,M_2)^{2k-1}
  \leq \left(\frac 1 2 \cos(M_1,M_2)+\frac 1 2\right)^{2k-1}\\
  &\leq \left(\frac 1 2 \cos(M_1,M_2)+\frac 1 2\right)^k
  =\left\| \left(\frac{P_{M_1}+P_{M_2}}{2}\right)^k - P_M\right\|,
\end{align}
where the inequalities are strict whenever $\cos(M_1,M_2)<1$. This somehow explains why, in general, the alternating projection method is indeed faster than its simultaneous variant whenever we have linear convergence. The numerical verification of this observation can be found, for example, in \cite[Fig. 1, p. 1071]{CensorChenCombettesDavidiHerman2012}.
\end{remark}

Next, we recall the following fact.

\begin{fact}\label{th:equiv}
Let $M_1,\ldots,M_r\subseteq\mathcal H$ be closed linear subspaces and let $M:=\bigcap_i M_i$. Moreover, let $\mathbf C$ and $\mathbf D$ be as in Example \ref{ex:Pierra1}. The following conditions are equivalent:
\begin{enumerate}[(i)]
  \item $\sum_{i=1}^r M_i^\perp$ is closed.
  \item $\cos(M_1,\ldots,M_r)<1$ (subspaces are not aligned).
  \item $\left\|\frac 1 r \sum_{i=1}^r P_{M_i} - P_M\right\|<1$.
  \item $\|P_{\mathbf D}P_{\mathbf C}-P_{\mathbf C \cap\mathbf D}\|<1$.
  \item $\cos(\mathbf C,\mathbf D)<1$.
  \item $\mathbf C^\perp+\mathbf D^\perp$ is closed in $\mathcal H^r$.
\end{enumerate}
\end{fact}

\begin{proof}
By \eqref{eq:AKW} applied to $\mathbf C$ and $\mathbf D$, and \eqref{BGM}, we have (ii) $\Leftrightarrow$ (iii) $ \Leftrightarrow$ (iv) $\Leftrightarrow$ (v). In order to complete the proof it suffices to show that (i) $\Leftrightarrow$ (iii) holds. The equivalence (v) $\Leftrightarrow$ (vi) will follow by using a similar argument, but in the product space $\mathcal H^r$.

Assume that (i) holds. Then, by Theorem \ref{th:dichotomy} (i), there are $c>0$ and $q\in (0,1)$ such that $\|T^k-P_M\|\leq c q^k$, where $T:=\frac 1 r \sum_{i=1}^r P_{M_i}$. This implies that $\|T^k-P_M\|<1$ for some integer $k\geq 1$. Since, by Theorem \ref{th:norm}, $\|T^k-P_M\|=\|T-P_M\|^k$, we conclude that $\|T-P_M\|<1$ too.

Now assume that (iii) holds and (i) does not, that is, $\sum_{i=1}^r M_i^\perp$ is not closed.  By Theorem \ref{th:dichotomy}, $T^k$ converges arbitrarily slowly to $P_M$. This is in contradiction with assumption (iii), in view of which $T^k$ converges linearly to $P_M$. This completes the proof.
\end{proof}

\begin{remark}
The equivalence (i) $\Leftrightarrow$ (ii) from Fact \ref{th:equiv} can be found in  \cite[Theorem 4.1, (3) $\Leftrightarrow$ (11)]{BadeaGrivauxMuller2011}. The argument follows from the fact that $\cos(M_1,\ldots,M_r)<1$ if and only if the family $\{M_1,\ldots,M_r\}$ is linearly regular \cite[Proposition 3.9]{BadeaGrivauxMuller2011}, which, by \cite[Theorem 5.19]{BauschkeBorwein1996}, is equivalent to (i). As we have already mentioned in the Introduction, for the case of $r=2$, an independent proof can be found in \cite[Theorem 9.35]{Deutsch2001}.
\end{remark}

We can now extend Theorems \ref{th:dichotomy} and \ref{th:dichotomyBS} in the case of the simultaneous projection method.

\begin{theorem}[Dichotomy with optimal error bound]\label{th:dichLin}
Let $M_1,\ldots,M_r\subseteq\mathcal H$ be closed linear subspaces, $M:=\bigcap_i M_i$ and let $T:=\frac 1 r \sum_{i=1}^r P_{M_i}$. Then exactly one of the following two statements holds:
\begin{enumerate}[(i)]
  \item $\sum_{i=1}^r M_i^\perp$ is closed. Then $T^k$ converges linearly to $P_M$ and
      \begin{equation}
      q=\frac{r-1}{r}\cos(M_1,\ldots,M_r)+\frac 1 r
      \end{equation}
      is the smallest possible number, independent of $x$, in the set of all $q\in(0,1)$ such that
      \begin{equation}
      \left\|T^k(x) - P_M(x)\right\|
      \leq q^k \|x\|
      \end{equation}
      for all $x\in\mathcal H$.
  \item $\sum_{i=1}^r M_i^\perp$ is not closed. Then $T^k$ converges arbitrarily slowly to $P_M$. Moreover, there is a dense linear subspace $X\subseteq\mathcal H$ on which $T^k$ converges super-polynomially fast to $P_M$.
\end{enumerate}
\end{theorem}

\begin{proof}
If $\sum_{i=1}^r M_i^\perp$ is closed, then linear convergence and the optimal error bound  follow from Theorem \ref{th:norm} and Fact \ref{th:equiv}.

Assume now that $\sum_{i=1}^r M_i^\perp$ is not closed. The arbitrarily slow convergence is an immediate consequence of Theorem \ref{th:dichotomy} (ii). In order to complete the proof, we have to show that the convergence is super-polynomially fast on some dense linear subspace $X\subseteq\mathcal H$. To this end, we apply the alternating projection formalization discussed in Example \ref{ex:Pierra1}.

Indeed, by Fact \ref{th:equiv}, $\mathbf C^\perp+\mathbf D^\perp$ is not closed in $\mathcal H^r$. Consequently, by Theorem \ref{th:dichotomyBS}, there is a dense linear subspace $\mathbf X\subseteq \mathcal H^r$ on which $(P_{\mathbf C}P_{\mathbf D})^k$ converges super-polynomially fast to $P_{\mathbf C\cap\mathbf D}$. Note that since $P_{\mathbf D}$ is nonexpansive,\ for each $\mathbf x\in\mathcal H^r$, we have
\begin{align}\nonumber
\|(P_{\mathbf D}P_{\mathbf C})^kP_{\mathbf D}\mathbf (\mathbf x)-P_{\mathbf C\cap\mathbf D} (\mathbf x)\|
&=\|P_{\mathbf D}(P_{\mathbf C}P_{\mathbf D})^k\mathbf (\mathbf x)-P_{\mathbf D}P_{\mathbf C\cap\mathbf D} (\mathbf x)\|\\
&\leq \|(P_{\mathbf C}P_{\mathbf D})^k\mathbf (\mathbf x)-P_{\mathbf C\cap\mathbf D} (\mathbf x)\|.
\end{align}
Consequently, for every $\mathbf x \in \mathbf X$ and $\alpha>0$, we have
\begin{equation}
\frac{\|(P_{\mathbf D}P_{\mathbf C})^kP_{\mathbf D}\mathbf (\mathbf x)-P_{\mathbf C\cap\mathbf D} (\mathbf x)\|}{k^{-\alpha}}
 \longrightarrow 0 \quad \text{ as } k\rightarrow\infty.
\end{equation}
This implies that $(P_{\mathbf D}P_{\mathbf C})^k$ converges super-polynomially fast to $P_{\mathbf C\cap\mathbf D}$ on $P_{\mathbf D}(\mathbf X)$. On the other hand, since $P_{\mathbf D}(\mathbf X)\subseteq \mathbf D$, we can define
\begin{equation}
X:=\{x\in\mathcal H\mid \mathbf x=(x,\ldots,x)\in P_{\mathbf D}(\mathbf X)\}.
\end{equation}
Observe that $X$ is a linear subspace of $\mathcal H$ because $P_{\mathbf D}(\mathbf X)$ is a linear subspace of $\mathcal H^r$, where the latter fact follows from the linearity of $P_{\mathbf D}$. Moreover, by \eqref{eq:Pierra}, for each $x\in X$ and $\alpha>0$, we have
\begin{equation}
\frac{\|T^k(x) - P_M(x)\|}{k^{-\alpha}}
= \frac{\|(P_{\mathbf D}P_{\mathbf C})^k (\mathbf x) - P_{\mathbf C\cap\mathbf D} (\mathbf x)\|}{k^{-\alpha}} \longrightarrow 0 \quad \text{ as } k\rightarrow\infty,
\end{equation}
where $\mathbf x=(x,\ldots,x)$. Consequently, $T^k$ converges super-polynomially fast to $P_M$ on $X$. It remains to prove that $X$ is dense in $\mathcal H$ or, equivalently, that $P_{\mathbf D}(\mathbf X)$ is dense in $\mathbf D$. Note that the second statement follows from the continuity of the metric projection $P_\mathbf D$ as we now show. Indeed, let $\mathbf x \in \mathbf D$. Since $\mathbf X$ is dense in $\mathcal H^r$, there is a sequence $\{\mathbf x_k\}_{k=0}^\infty\subseteq \mathbf X$ such that $\mathbf x_k\rightarrow \mathbf x$ and the above-mentioned continuity yields $P_{\mathbf D}(\mathbf x_k)\rightarrow P_{\mathbf D}(\mathbf x)=\mathbf x$. This completes the proof.
\end{proof}

\begin{corollary}[Affine subspaces]\label{th:dichAff}
Let $V_1,\ldots,V_r\subseteq\mathcal H$ be closed affine subspaces and assume that $V:=\bigcap_i V_i\neq\emptyset$. Moreover, let $T:=\frac 1 r \sum_{i=1}^r P_{V_i}$. Then exactly one of the following two statements holds:
\begin{enumerate}[(i)]
  \item $\sum_{i=1}^r (V_i-V_i)^\perp$ is closed. Then $T^k$ converges linearly to $P_V$ and
      \begin{equation}\label{th:dichAff:1}
      q=\frac{r-1}{r}\cos(V_1,\ldots,V_r)+\frac 1 r
      \end{equation}
      is the smallest possible number, independent of $x$, in the set of all $q\in(0,1)$ such that
      \begin{equation}\label{th:dichAff:2}
      \left\|T^k(x) - P_V(x)\right\|
      \leq q^k \|x-P_V(0)\|
      \end{equation}
      for all $x\in \mathcal H$, where $\cos(V_1,\ldots,V_r):=\cos(V_1-V_1,\ldots,V_r-V_r)$.
  \item $\sum_{i=1}^r (V_i-V_i)^\perp$ is not closed. Then $T^k$ converges arbitrarily slowly to $P_V$. Moreover, there is a dense affine subspace $Y\subseteq\mathcal H$ on which $T^k$ converges super-polynomially fast to $P_V$.
\end{enumerate}
\end{corollary}
\begin{proof}
The proof is based on the translation formula
\begin{equation}\label{proof:dichAff:translation}
P_C(x)=P_{C-v}(x-v)+v,
\end{equation}
which holds true for every closed and convex set $C\subseteq\mathcal H$, $x,v\in\mathcal H$. Note that for any $x\in\mathcal H$ and $v\in V$, we have $V_i=M_i+v$ and $V=M+v$, where $M_i:=V_i-V_i$ and $M:=V-V$ are closed linear subspaces. This holds, in particular, for $v=P_V(0)$ which, when combined with an induction argument, leads to
\begin{equation}\label{proof:dichAff:3}
\left(\frac 1 r \sum_{i=1}^rP_{V_i}\right)^k(x)-P_V(x)
=\left(\frac 1 r \sum_{i=1}^rP_{M_i}\right)^k(x-v)-P_M(x-v).
\end{equation}

If $\sum_i M_i^\perp$ is closed, then, by Theorem \ref{th:dichLin} (i) and \eqref{proof:dichAff:3}, estimate \eqref{th:dichAff:2} holds with $q$ defined as in \eqref{th:dichAff:1}. If there were another $0<q<\frac{r-1}{r}\cos(V_1,\ldots,V_r)+\frac 1 r$ for which \eqref{th:dichAff:2} holds, then, by \eqref{proof:dichAff:3}, this would imply that
\begin{equation}
\left\|\left(\frac 1 r \sum_{i=1}^r P_{M_i}\right)^k - P_M\right\|
\leq q^k
< \left(\frac{r-1}{r}\cos(V_1,\ldots,V_r)+\frac 1 r\right)^k,
\end{equation}
which is impossible in view of Theorem \ref{th:norm}.

If $\sum_i M_i^\perp$ is not closed, then, by Theorem \ref{th:dichLin} (ii), $\left(\frac 1 r \sum_{i=1}^r P_{M_i}\right)^k$ converges arbitrarily slowly to $P_M$. This implies, by \eqref{proof:dichAff:3}, that $\left(\frac 1 r \sum_{i=1}^r P_{V_i}\right)^k$ also converges arbitrarily slowly to $P_V$. Moreover, by Theorem \ref{th:dichLin} (ii), there is a dense subspace $X\subseteq\mathcal H$ on which $\left(\frac 1 r \sum_{i=1}^r P_{M_i}\right)^k$ converges super-polynomially fast to $P_M$. It is easy to see that, by \eqref{proof:dichAff:3}, $\left(\frac 1 r \sum_{i=1}^r P_{V_i}\right)^k$ converges super-polynomially fast to $P_V$ on $Y:=X+v$, which in its turn is a dense affine subspace of $\mathcal H$.
\end{proof}

\begin{remark}(Cyclic projection method)
Let $M_1,\ldots,M_r\subseteq\mathcal H$ be closed linear subspaces and let $M:=\bigcap_i M_i$. By \cite[Theorem 1]{KayalarWeinert1988}, $P_{M_r\cap M^\perp}\ldots P_{M_1\cap M^\perp}=P_{M_r}\ldots P_{M_1}-P_M$. Moreover, by Lemma \ref{th:basic} applied to $T:=P_{M_r}\ldots P_{M_1}$, we see that
\begin{equation}
(P_{M_r}\ldots P_{M_1}  )^k-P_M
= (P_{M_r}\ldots P_{M_1}  -P_M)^k
= (P_{M_r\cap M^\perp}\ldots P_{M_1\cap M^\perp})^k,
\end{equation}
which leads to the following (not necessarily optimal) error bound:
\begin{equation}
\left\| (P_{M_r}\ldots P_{M_1}  )^k(x)-P_M(x)\right\|
\leq \|P_{M_r\cap M^\perp}\ldots P_{M_1\cap M^\perp}\|^k \|x\|.
\end{equation}
If we assume that $\sum_{i=1}^r M_i^\perp$ is closed, then, by \cite[Theorem 4.1]{BadeaGrivauxMuller2011}, $\|P_{M_r\cap M^\perp}\ldots P_{M_1\cap M^\perp}\|<1$ and the convergence is indeed linear. The above estimate can be found, for example, in \cite[Remark 5.5.3]{Bauschke1996}, \cite[Lemma 11.58]{DeutschHundal2011} and \cite[Lemma 9.2]{DeutschHundal2015}. Note that for closed affine subspaces $V_1,\ldots,V_r$ with nonempty intersection $V:=\bigcap_iV_i$ we can, similarly to \eqref{proof:dichAff:3}, derive the equality
\begin{equation}
(P_{V_r}\ldots P_{V_1} )^k(x)-P_V(x)
=(P_{V_r-V_r}\ldots P_{V_1-V_1} )^k(x)-P_{V-V}(x-v)
\end{equation}
for every $v\in V$. This leads to the following error bound:
\begin{equation}
\left\| (P_{V_r}\ldots P_{V_1}  )^k(x)-P_V(x)\right\|
\leq \left\| \prod_{i=1}^r P_{(V_i-V_i)\cap (V-V)^\perp}\right\|^k \|x-P_V(0)\|,
\end{equation}
which yields linear convergence whenever $\sum_{i=1}^r(V_i-V_i)^\perp$ is closed. The above estimate can be found, for example, in \cite[Theorem 6.6]{DeutschHundal2010}.
\end{remark}

\section*{Appendix}
In this section we present an argument showing that Theorem \ref{th:dichotomyBS} which is stated here for a real Hilbert space indeed follows from \cite[Theorem 4.3]{BadeaSeifert2016}, which is shown to be true for a complex Hilbert space. The proof is based on a hand-written note which has kindly been provided to us by Catalin Badea and David Seifert.

\begin{proof}
Let $\mathcal H_{\mathbb C}=\mathcal H \times \mathcal H$ be a complexification of the real Hilbert space $\mathcal H$ equipped with the scalar product $\langle \cdot, \cdot \rangle_{\mathbb C}\colon \mathcal H_{\mathbb C} \times \mathcal H_{\mathbb C} \rightarrow \mathbb C$ defined by
\begin{equation}\label{}
  \langle x+iy, x'+iy' \rangle_{\mathbb C}:=\langle x,x'\rangle + \langle y,y'\rangle +i(\langle x,y'\rangle - \langle x',y\rangle)
\end{equation}
where the induced norm satisfies $\|x+iy\|^2_{\mathbb C}=\|x\|^2+\|y\|^2$. Consider $\tilde{M}:=M\times M$ and $\tilde M_j:=M_j\times M_j$, $j=1,\ldots,r$, which are closed linear subspaces of $\mathcal H_{\mathbb C}$. Moreover, define $\tilde P$ and $\tilde P_j$ as the complexification of $P=P_M$ and $P_j=P_{M_j}$, that is, $\tilde P (x+iy):=Px+iPy$ and $\tilde P_j (x+iy):=P_jx+iP_jy$. Observe that $\tilde P$ and $\tilde P_j$ are linear, idempotent and, in fact, are orthogonal projections onto $\tilde M$ and $\tilde M_j$, respectively. Indeed, for any $a+ib\in\mathbb C$ and $x+iy\in\mathcal H_{\mathbb C}$, by the linearity of $P$, we have
\begin{align} \nonumber
  (a+ib) \tilde P(x+iy) & = (a+ib)(Px+iPy) = aPx-bPy + i (aPy+bPx)\\
  & = P(ax-by)+i P(ay+bx) = P((a+ib)(x+iy)),
\end{align}
which shows that $\tilde P$ is linear. Since $P=P^2$, it easily follows that $\tilde P=\tilde P^2$. Moreover, for each $u+iv\in \tilde M$, since $P$ is an orthogonal projection, we get
\begin{align}\label{}\nonumber
   \langle x+iy - \tilde P (x+iy), u+iv \rangle_{\mathbb C} &
   = \langle x-Px+iy-iPy, u+iv \rangle_{\mathbb C}\\ \nonumber
   & = \langle x-Px, u\rangle + \langle y-Py, v\rangle\\
   & + i \langle x-Px, v\rangle -i \langle y-Py, u\rangle =0,
\end{align}
which shows that $\tilde P$ is an orthogonal projection too. The same argument can be repeated for each $\tilde P_j$. Finally, observe that
\begin{equation}\label{}
\tilde M^\perp= M^\perp\times M^\perp \quad \text{and} \quad
\tilde M_i^\perp= M_i^\perp\times M_i^\perp.
\end{equation}
Indeed, it is easy to see that $\tilde M^\perp \supseteq M^\perp\times M^\perp$. Moreover, if $x+iy\in \tilde M^\perp$ and $m+im \in \tilde M$, then, by the definition of $\langle \cdot, \cdot \rangle_{\mathbb C}$,
\begin{equation}
  \langle x+y,m \rangle = \langle x-y,m \rangle= 0.
\end{equation}
This implies that $x+y$, $x-y \in M^\perp$ and consequently, $x=\frac 1 2 ((x+y)+(x-y))\in M^\perp$ and $y=\frac 1 2 ((x+y)-(x-y))\in M^\perp$.

By assumption and Fact \ref{th:equiv}, we have
\begin{equation}
\cos(M_1,\ldots,M_r)=\frac 1{r-1}\sup\left\{
\frac{\sum_{j\neq k}\langle m_j,m_k\rangle}{\sum_j \|m_j\|^2}
\mid m_j\in M_j\cap M^\perp,\ \sum_j\|m_j\|^2\neq 0
\right\}=1.
\end{equation}
We claim that $\cos(\tilde M_1,\ldots,\tilde M_r)=1$. Indeed, for each $\varepsilon>0$, there are $m_j\in M_j\cap M^\perp$ such that $\sum_j \|m_j\|^2\neq 0$ and
\begin{equation}\label{}
  \frac 1{r-1}\frac{\sum_{j\neq k}\langle m_j,m_k\rangle}{\sum_j \|m_j\|^2} \geq 1-\varepsilon.
\end{equation}
Moreover, by setting $\tilde m_j:=m_j+i0$, we see that $\tilde m_j\in \tilde M_j\cap\tilde M^\perp$, $\langle \tilde m_j, \tilde m_k\rangle = \langle  m_j, m_k\rangle$ and, in particular, $\sum\|\tilde m_j\|^2=\sum\|m_j\|^2\neq 0$. Consequently, $1-\varepsilon\leq\cos(\tilde M_1,\ldots,\tilde M_r)\leq1$ holds for each $\varepsilon >0$, which shows that $\cos(\tilde M_1,\ldots,\tilde M_r)=1$, as claimed.

Now consider $\tilde T:=\tilde P_r\ldots \tilde P_1$, which is, in fact, the complexification of $T$, that is, $\tilde T (x+iy)=Tx+iTy$ for each $x+iy\in\mathcal H_{\mathbb C}$. By applying \cite[Theorem 4.3]{BadeaSeifert2016} to $\tilde M, \tilde M_1,\ldots,\tilde M_r\subseteq \mathcal H_{\mathbb C}$ and $\tilde T$, we get that there is a dense subspace $\tilde X\subseteq\mathcal H_{\mathbb C}$ on which $\tilde T^k$ converges super-polynomially fast to $\tilde P$. Consequently, for each $x+iy\in\tilde X$ and $\alpha >0$, we have
\begin{align}\label{eq:pr:dichotomyBS}\nonumber
  \frac{\|T^k(x) - P(x)\|}{k^{\alpha}} &\leq
  \frac{\|T^k(x) - P(x)\|^2+ \|T^k(y)-P(y)\|^2}{k^{2\alpha}}\\
  &= \frac{\|\tilde T^k (x+iy)-\tilde P(x+iy)\|^2}{k^{2\alpha}}
  \longrightarrow 0 \quad \text{ as } k\rightarrow\infty.
\end{align}
Consider $X:=\{x\in\mathcal H\mid x+iy\in\tilde X \text{ for some } y\in\mathcal H\}$ and observe that since $\tilde X$ is a dense linear subspace of $\mathcal H_{\mathbb C}$, the space $X$ is a dense linear subspace of $\mathcal H$. Moreover, the super-polynomial convergence of $T^k$ to $P$ on $X$ follows from \eqref{eq:pr:dichotomyBS}.
\end{proof}

\vspace{2em}
\noindent\textbf{Acknowledgements.} We would like to thank Catalin Badea and David Seifert for kindly providing us with the complexification argument presented in the Appendix. We are also very grateful to an anonymous referee for pertinent comments and helpful suggestions.

\vspace{2em}
\noindent\textbf{Funding.} This research was supported in part by the Israel Science Foundation (Grant~389/12), the Fund for the Promotion of Research at the Technion and by the Technion General Research Fund.

\small
\bibliographystyle{siam}
\bibliography{references}

\end{document}